\documentclass[11pt,a4paper,english,reqno]{amsart}
\usepackage{amsmath,amssymb,amsfonts,epsfig,mathrsfs}
\usepackage[T1]{fontenc}

\usepackage{color}
\usepackage{array}
\usepackage{amsthm}
\usepackage{amstext}
\usepackage{graphicx}
\usepackage{setspace}
\usepackage[margin=2.5cm]{geometry}
\usepackage{bbm}
\usepackage{color}
\usepackage{enumitem}
\allowdisplaybreaks[4]

\usepackage{amscd,psfrag}
\usepackage{yhmath}
\usepackage[mathscr]{eucal}
\usepackage{slashed}

\makeatletter
\pdfpageheight\paperheight
\pdfpagewidth\paperwidth

\usepackage{epstopdf}
\usepackage{chngcntr}
\counterwithin{figure}{section}
\usepackage{mathrsfs}

\usepackage{indentfirst}	

\usepackage[normalem]{ulem}
\theoremstyle{plain}
\newtheorem{definition}{Definition}[section]
\newtheorem{theorem}[definition]{Theorem}
\newtheorem*{theorem*}{Theorem}

\newtheorem{remark}[definition]{Remark}

\newtheorem*{remark*}{Remark}
\newtheorem*{sideremark*}{Side Remark}\newtheorem*{mt*}{Main Theorem}

\newtheorem*{claim*}{Claim}
\newtheorem*{q*}{Question}
\newtheorem{lemma}[definition]{Lemma}
\newtheorem{corollary}[definition]{Corollary}
\newtheorem*{corollary*}{Corollary}
\newtheorem*{proposition*}{Proposition}

\newcommand{\R}{\mathbb{R}}

\newcommand{\p}{\partial}
\newcommand{\e}{\epsilon}

\newcommand{\map}{\rightarrow}

\newcommand{\M}{\mathcal{M}}

\newcommand{\two}{{\rm II}}

\newcommand{\stwo}{{\mathbb{S}^2}}
\newcommand{\gbar}{{\overline{g}}}
\newcommand{\jt}{{\widetilde{J}}}
\newcommand{\A}{\mathcal{A}}

\allowdisplaybreaks[4]

\def\XXint#1#2#3{{\setbox0=\hbox{$#1{#2#3}{\int}$ }
\vcenter{\hbox{$#2#3$ }}\kern-.6\wd0}}

\newcommand{\E}{{\mathbf{E}}}
\newcommand{\V}{{\mathbf{V}}}
\newcommand{\U}{{\mathbf{U}}}
\newcommand{\DD}{{\mathbf{D}}}
\newcommand{\J}{{\mathscr{J}}}

\numberwithin{equation}{section}
\numberwithin{figure}{section}

\title{The Weyl problem of  isometric immersions  revisited}

\author{Siran Li}
\address{Siran Li: Department of Mathematics, Rice University, MS 136
P.O. Box 1892, Houston, Texas, 77251, USA}

\email{\texttt{Siran.Li@rice.edu}}



\date{\today}

\begin{document}

\maketitle

\begin{abstract}
We revisit the classical problem by Weyl, as well as its generalisations,  concerning the isometric immersions of $\stwo$ into simply-connected $3$-dimensional Riemannian manifolds with non-negative Gauss curvature. A sufficient condition is exhibited for the existence of global $C^{1,1}$-isometric immersions. Our developments are based on the framework \`{a} la Labourie \cite{l} of analysing isometric immersions via $J$-holomorphic curves. We obtain along the way a generalisation of a well-known theorem due to Heinz and Pogorelov.

\end{abstract}

\section{Introduction}\label{sec: Introduction}

\subsection{The Weyl problem}
We are concerned with the problem of the existence of isometric immersions of a surface $(\Sigma,g)$ homeomorphic to $\stwo$ with intrinsic curvature $\geq K_0$ into a $3$-dimensional simply-connected Riemannian manifold $(\M,\gbar)$ with sectional curvature $\leq K_0$; here $K_0$ is any finite real number. This classical problem was first investigated by Weyl \cite{w} in 1916, for $(\M,\gbar)$ being the Euclidean $3$-space and $K_0=0$. It has played a significant r\^{o}le in the development of geometric analysis and nonlinear PDEs. See \cite{hh} for a comprehensive exposition. 

For $\M=\R^3$, Lewy \cite{l} in 1938 solved the problem for real-analytic metric $g$ with strictly positive Gauss curvature. Nirenberg proved this result for $g\in C^4$ in his seminal 1953 paper \cite{n}. The case $g\in C^3$ was later settled by Heinz \cite{h} in 1962. Using different methods, Aleksandrov and Pogorelov \cite{a,p} obtained generalised solutions to the Weyl problem by considering limits of convex polyhedra.

For $\M=\R^3$ and metric $g$ with non-negative Gauss curvature, Guan--Li \cite{gl} proved the existence of $C^{1,1}$-isometric immersions for $g \in C^4$; also see Hong--Zuily \cite{hz}. The case of $\M=\mathbb{H}^3$ was settled by Pogorelov \cite{p2} for Gauss curvature $K >-1$, and by Lin--Wang \cite{lw} for $K\geq -1$; {\it cf.} also Chang--Xiao \cite{cx}. For the existence results of general ambient manifolds $(\M,\gbar)$ other than space forms, we refer to Pogorelov \cite{p} and  recent works by Guan--Lu \cite{glu} and Li--Wang \cite{li-wang}.

On the other hand,  Burago--Shefel' and Iaia \cite{bs,i} constructed interesting examples for a topological two-sphere $(\Sigma,g)$ whose metric is real-analytic and Gauss curvature is positive everywhere except at one point, but does \emph{not} admit $C^3$-global isometric embeddings into $\R^3$. Thus, there are obstructions to the existence of isometric embeddings/immersions of high regularity.

\subsection{The degenerate case}
%This paper is motivated by the above examples as well as the following problem of Guan--Li (\cite{gl}, p.333 Question 2):
%\begin{quotation}\emph{What are the sufficient conditions $($even necessary and sufficient conditions$)$ on the metric with nonnegative Gauss curvature which give rise to a smooth isometric embedding into $(\R^3,\delta)$? $[\delta$ is the Euclidean metric.$]$}\end{quotation}

In this paper, we  consider a smooth surface $(\Sigma,g)$ homeomorphic to $\stwo$ and a $3$-dimensional simply-connected manifold $(\M,\gbar)$ --- not required to be a space-form --- such that the Gauss curvature of $g$ is less than or equal to the sectional curvature of $\gbar$. We establish one sufficient condition for the existence of a $C^{1,1}$-isometric immersion $f:(\Sigma,g)\to (\M,\gbar)$, formulated in terms of  degeneracy/blowup rates for the principal curvatures of approximate families of elliptic embeddings.

In passing, we remark that in the nice paper \cite{hl}, Han--Lin obtained a sufficient and necessary condition for the existence of $C^\infty$-isometric embeddings for a family of metrics on $\mathbb{T}^2$ into $(\R^3,\delta)$. The method in \cite{hl} relies crucially on rigidity results for surfaces of sign-changing curvatures.

\subsection{The main theorem}
Our main result of this paper is the following:

\begin{theorem}\label{thm: main}
Let $(\Sigma,g)$ be a surface homeomorphic to $\stwo$ with intrinsic curvature greater than or equal to a real number $K_0$. Let $(\M,\gbar)$ be a $3$-dimensional simply-connected Riemannian manifold whose sectional curvature is less than or equal to $K_0$. Assume that $g,\gbar \in C^3$. Then one of the following   scenarios holds ---
\begin{enumerate}
\item
{\bf Either} there exists a degenerate-elliptic $C^{1,1}$-isometric immersion $f: (\Sigma,g) \to (\M,\gbar)$;
\item
{\bf or} there are a sequence of smooth metrics $\{g^\e\}$ converging to $g$ in the Lipschitz norm as $\e \map 0$, and a sequence $\{f^\e\}$ of smooth $\e$-elliptic isometric immersions of $g^\e$ into $(\M,\gbar)$, such that at any point where the smaller principal curvature $\kappa^\e_1 \to 0$,  the larger principal curvature $\kappa_2^\e$ must blow up at a rate no faster than $\mathcal{O}({1}\slash{\sqrt[3]{\kappa_1^\e}})$.
\end{enumerate}

Moreover, in Case (2) above we can bound  %or in the case where there exists a degenerate-elliptic $C^\infty$-isometric immersion $f$, 
\begin{equation}\label{bound--new}
\kappa_2^\e(z) \leq \frac{C_0}{\sqrt[3]{\kappa_1^\e(z)}}\qquad \text{ for all } z \in \Sigma \sim (\kappa_1^\e)^{-1}\{0\},
\end{equation}
where $C_0$ depends only on $\|g\|_{C^3}$ and $\|\gbar\|_{C^3}$. 
\end{theorem}
%\begin{remark}
%The same bound~\eqref{bound--new} also holds when there is a degenerate-elliptic $C^\infty$-isometric immersion $f$,  with $\{\kappa_1^\e,\kappa^2_\e\}$ replaced respectively by $\{\kappa_1, \kappa_2\}$, the principal curvatures of $f$.
%\end{remark}

Our developments are largely based on the framework laid down by Labourie \cite{l}, and our notations and nomenclatures closely  follow \cite{l}. In particular, we adopt the following
\begin{definition}
An isometric immersion $f:(\Sigma,g) \map (\M,\gbar)$ is said to be {\bf $\e$-elliptic} if the Gauss curvature of $f$ as in Eq.~\eqref{K, k} satisfy $K\geq \e>0$ everywhere on $\Sigma$. The immersion $f$ or the immersed surface $f(\Sigma)$ is {\bf elliptic} if $f$ is $\e$-elliptic for some $\e$. It is {\bf degenerate-elliptic} if $K\geq 0$ everywhere on $\Sigma$. Throughout, $K$ denotes the {\bf Gauss curvature of $f$}, namely the difference between the Gauss curvatures of $\Sigma$ and $f_\#(T\Sigma)$.
\end{definition}

We use $f_\#$ and $f^\#$ to denote the pushforward and pullback under $f$, respectively.

\subsection{Roadmap}
To illustrate the point of  Theorem~\ref{thm: main} and the strategy for its proof, the following discussions are presented.

In order to find an isometric immersion $(\Sigma,g)\to(\M,\gbar)$, one natural approach is to first approximate the metric $g$ by smooth metrics $\{g^\e\}$ with Gauss curvatures $K_\e\geq \e>0$. This can be done, for example, by a conformal change of metrics together with a mollification (see \cite{gl}; also see the proof of Theorem~\ref{thm: main} below). Then, by the existence results for strictly positively curved metrics ({\it cf.} Pogorelov \cite{p}, Labourie \cite{l}, Lin \cite{lin}, etc.), each $g^\e$ admits an smooth isometric immersion $f^\e$ into $(\M,\gbar)$. So it remains to investigate if one can pass to the limits for $\{f^\e\}$. %In other words, $\{f^\e\}$ are elliptic regularisations of $f$.

Here enter the crucial insights by Labourie \cite{l}. Throughout this paragraph let us drop all the super-/subscripts $\e$ for notational simplicity. 
Let $f$ be an $\e$-elliptic isometric immersion. Its $1$-jet can be viewed as a pseudo-holomorphic map from $\Sigma$ into the fibre bundle 
\begin{equation*}
\begin{matrix}
\E:={\rm Isom}(T\Sigma,T\M)\\
\Big\downarrow \pi\\
\Sigma \times \M.
\end{matrix}
\end{equation*}
For $\xi \in \E$, the tangent space $T_\xi \E$ splits into $\V \oplus \U$, of which the important component is 
\begin{align}\label{Xi, V}
\V = \Big\{ \Xi(u,v):= \big(u,\xi(u),\,K\xi(v)\big):\,u,v\in T\Sigma \Big\},
\end{align}
equipped with the almost complex structure $J|\V: \Xi(u,v) \mapsto \Xi(v,-u)$. 
For an isometric immersion $f^\e:(\Sigma,g)\map(\M,\gbar)$ we take $\xi = df = f_\#$. Throughout $K$ denotes the (relative) Gauss curvature, {\it i.e.},
\begin{equation}\label{K, k}
K:= K(f_\#T\Sigma) - K(\Sigma) \equiv k^2.
\end{equation}
(We use the symbol $k$ in accordance with \cite{l}.) Then, for the $1$-jet of $f$,
\begin{equation*}
j_1f(\Sigma) \subset \V \subset \{\xi\in \E:\,k(\xi)>0\}.
\end{equation*}
In addition, $\V$ is \emph{calibrated} in the following sense: there is a $1$-form $\varphi$ defined on some neighbourhood of $j_1f(\Sigma)$, such that
\begin{equation*}
d\varphi\big(x,J|\V(x)\big)>0.
\end{equation*}
See \cite{l}, 2.10; also {\it cf.} Harvey--Lawson \cite{calibre} for  calibration. Thanks to the calibration $\varphi$, it follows from \emph{le lemme de Schwarz \`{a} la Gromov} (see \cite{g}; McDuff--Salamon \cite{ms}) that if $j_1f(\Sigma)$ is precompact, then $j_1f$ is smooth with uniformly bounded derivatives of all orders. Note that by \cite{l}, 2.6, the compactness in $\V$ is understood with respect to the following Hermitian metric $\mu$:
\begin{equation}\label{herm metric}
\mu\Big(\Xi(u_1,v_1),\Xi(u_2,v_2)\Big) := k g(u_1,u_2) + kg(v_1,v_2),
\end{equation}
where $g$ is the metric on $\Sigma$ in consideration. 

For all the above to hold, we need $(\M,\gbar)$ to be  simply-connected and to have sectional curvature $\leq K_0$. Nevertheless, it does not have to be a space form.

Let us apply the above arguments to $\{g^\e\}$ to get smooth isometric immersions $\{f^\e\}$. When the \emph{extrinsic geometries} --- namely, the mean curvatures $H_\e$ --- of $\{f^\e\}$ are uniformly bounded independent of $\e$,  by Arzel\`{a}--Ascoli's theorem one may pass to the limits to obtain a $C^{1,1}$-isometric immersion.
%combining the Arzel\`{a}--Ascoli theorem and le lemme du Schwarz \`{a} la Gromov ({\it e.g.}, via an adaptation of the arguments for 1.6 Lemme in Labourie \cite{l}), one may infer that $\{f^\e\}$ converges to a smooth isometric immersion.

It remains to consider the case when the mean curvatures are not uniformly bounded. The blowup of mean curvatures only occurs  in the limiting process $\e\map 0$, {\it i.e.}, when the isometrically immersed surfaces $f^\e(\Sigma) \subset \M$ lose strict ellipticity. In this case, the Gauss curvatures tend to zero, while the mean curvatures blow up to infinity.

Our crucial observation is Theorem~\ref{thm: weyl} below:  in the above degenerate scenarios, the product of mean curvatures and the square root of the Gauss curvatures remains bounded:
\begin{equation}\label{H,K pairing}
H_\e \cdot \sqrt{K_\e} \leq b_0.
\end{equation}
The constant $b_0$ depends only on the supremum of $K_\e$ over $\Sigma$, which is bounded by the supremum of the Gauss curvature $K$ of $(\Sigma,g)$.

The bound~\eqref{H,K pairing} is proved by PDE methods. In Section~\ref{sec: PDE} we recall a first-order PDE, namely Eq.~\eqref{PDE}, on mean curvature $H$, or on the inverse of mean curvature $W:=H^{-1}$, derived by Labourie in 2.13 Lemme,  \cite{l}. Then, in Section~\ref{sec: harnack}, by taking another exterior derivative to Eq.~\eqref{PDE} we get a second-order elliptic equation of the divergence form, whose lower-order terms satisfy good estimates. Such estimates allow us to prove a Harnack inequality for $W$, \emph{on the set where $W$ is less than a uniform constant times $k$.} Most importantly, this Harnack estimate is independent of the  parameter $\e$ for the elliptic regularisation $\{f^\e\}$.

With the aforementioned preparations,  we complete the proof of Theorem~\ref{thm: main} in Section~\ref{sec: proof}. In particular, Eq.~\eqref{H,K pairing} can be readily translated into a  comparison result between degeneracy/blowup rates of the two principal curvatures.

Our arguments also lead to new results in the non-degenerate case, {\it i.e.}, when $k^2>0$ strictly. The aforementioned Harnack  estimates allow us to deduce the existence of isometric immersions for $g$ and $\gbar$ only in $C^3$. This generalises the classical results of Heinz \cite{h} and Pogorelov \cite{p} (also see F.-H. Lin \cite{lin}) for the ambient manifold $\M$ being a space form.

\subsection{A sufficient condition}
To conclude the introduction, we paraphrase Theorem~\ref{thm: main} into a criterion for the existence of degenerate-elliptic isometric immersions. 

\begin{corollary}
Let $(\Sigma,g)$ be a  surface homeomorphic to $\stwo$ with curvature greater than or equal to a real number $K_0$. Let $(\M,\gbar)$ be a $3$-dimensional simply-connected Riemannian manifold whose sectional curvature is less than or equal to $K_0$. Suppose $g,\gbar \in C^3$.  Assume that for any smooth isometric immersions $\{f^\e\}$ of the $\e$-elliptic regularisations $\{g^\e\}$ of $g$, 
\begin{itemize}
\item
{\bf either} the mean curvatures $H_\e$ are uniformly $C^0$-bounded;
\item
{\bf or} on the set where $H_\e$ blows up in the limit, there holds
\begin{equation*}
\kappa_2^\e \cdot \sqrt[3]{\kappa_1^\e} \longrightarrow +\infty\qquad \text{ as } \e\to 0.
\end{equation*}
Then $\{f^\e\}$ converges to a $C^{1,1}$-degenerate-elliptic isometric immersion $f: (\Sigma,g) \map (\M,\gbar)$.
\end{itemize}
\end{corollary}

\begin{remark}
The results of this paper are global in nature, as the local existence of isometric immersions for $(\Sigma,g)$ into $(\R^3,\delta)$ with $K_g \geq 0$ is known  for sufficiently regular metrics. See C.-S. Lin \cite{cslin} and Han \cite{han}.
\end{remark}

\section{A first-order PDE for mean curvature}\label{sec: PDE}

In this section $f:(\Sigma,g) \to (\M,\gbar)$ is an $\e$-elliptic isometric immersion. For notational convenience here we drop the super-/subscript $\e$. Let $J,\jt \in {\rm End}(T\Sigma)$ be the almost complex structure on $\Sigma$ with respect to the second and the first fundamental forms, respectively. By the  ellipticity of $f$, the second fundamental form $\two$ is indeed a metric.

We shall view the mean curvature $H$ of $f$ as  defined on a subset of the $1$-jet bundle:
\begin{equation*}
H: \DD \equiv j^1f\circ \gamma (\Delta) \subset \J^1(\Sigma,\M) \longrightarrow \R,
\end{equation*}
where $\gamma$ is a conformal map from the unit disc $\Delta$ to an open subset of $\Sigma$, and $j^1f$ is the $1$-jet of the $\e$-elliptic isometric immersion $f$, namely that 
\begin{equation*}
j_1f(x):=\big(f(x),\,f_\#|_x=d_xf\big).
\end{equation*}

For the principal curvatures $\kappa_1$ and $\kappa_2$ with respect to the isometric immersion $f$, we have $H=\frac{\kappa_1+\kappa_2}{2}$ and $K=k^2=\kappa_1\kappa_2$. Recall that  $K \geq \e$ by $\e$-ellipticity. Let $\omega \in \A^1(T\Sigma)$ be the connection form associated to the principal directions corresponding to $\kappa_1$ and $\kappa_2$. Then $\pi_\Sigma^\#\omega \in \A^1( \J^1(\Sigma,\M))$. Throughout, $\A^p(X)$ denotes  differential $p$-forms over  bundle $X$.

One ingenious observation by Labourie (\cite{l}, 2.12, 2.13, and 3.6) is that, thanks to the Gauss--Codazzi equations of isometric immersions, we can derive a first-order PDE for $H$:
\begin{equation}\label{PDE, H}
dH \circ J = H\beta + \pi_\Sigma^\# \omega (H^2-4k^2).
\end{equation}
Here $\pi_\Sigma$ is the projection from $T\E$ onto $T\Sigma$, and $\pi^\#_\Sigma$ is the pullback operator under this projection. $\beta \in  \A^1( \J^1(\Sigma,\M))$  depends only on $\pi$, $k$, and operators $L_1$, $L_2$, where
\begin{align*}
L_1\big(\Xi(u,v)\big) := k \Xi(0,u)
\end{align*}
and
\begin{align*}
L_2\big(\Xi(u,v), \Xi(w,q)\big) := k\Xi\big(0,-J_0 \overline{R}_u(w)\big),
\end{align*}
where the almost complex structure $J_0$ is given by
\begin{equation*}
J_0(u) := \nu \wedge u
\end{equation*}
for $\nu$ being the outward unit normal vectorfield along $f(\Sigma) \subset \M$ and $\wedge$ being the cross product of vectorfields, and $\overline{R}_u$ is given by
\begin{equation*}
\overline{R}_u(v):=R(u,v)\nu + J_0R(u,J_0v)\nu,
\end{equation*} 
with $R$ being the Riemann curvature tensor of $(\Sigma,g)$.

\section{Weyl's estimate via Harnack}\label{sec: harnack}

This section is dedicated to the proof of the following ``dichotomy theorem''. As before, $H$ and $k^2$ are the mean and Gauss curvatures, respectively. Again, for notational convenience we shall drop the super-/subscripts $\e$ in this section.

\begin{theorem}\label{thm: weyl}
Let $(\Sigma,g)$ be a surface homeomorphic to $\stwo$ with curvature strictly larger than a real number $K_0$. Let $(\M,\gbar)$ be a $3$-dimensional simply-connected Riemannian manifold whose sectional curvature is less than or equal to $K_0$. Let $f: (\Sigma,g) \to (\M,\gbar)$ be an $\e$-elliptic isometric immersion. Then the following holds:

There are two finite numbers $a_0$ and $b_0$, with $a_0$ depending only on $\{\|g\|_{C^3}, \|\gbar\|_{C^3}\}$ and $b_0$ depending only on $\{\|g\|_{C^2}, \|\gbar\|_{C^2}\}$, such that for any $a \geq a_0$  exactly one of the following holds:
\begin{itemize}
\item
$\max_{x\in \Sigma} H(x) \leq a$;
\item
$H(x) > a$ at some point but, simultaneously, $H(x)k(x) \leq b_0$.
%\item
%$H(x) = +\infty$ for all $x \in \Sigma$.
\end{itemize}
%In the above, $H$ and $k^2$ are the mean and Gauss curvatures, respectively.
\end{theorem}

\begin{remark}
We may view Theorem~\ref{thm: weyl} as a variant of the  \emph{Weyl's estimate}, which bounds the extrinsic geometry by the intrinsic geometry. 

Notably, in \cite{gl} Guan--Li obtained for the isometric embedding of $(\Sigma,g)$ into $(\R^3,\delta)$ that
\begin{equation*}
\max_\Sigma H \lesssim \sqrt{\max_\Sigma\Big(K^2 -\frac{3}{2}\Delta_g K \Big)},
\end{equation*}
which involves up to two derivatives of $K$, namely $C^4$-bounds for $g$. This leads to the proof for the existence of $C^{1,1}$-isometric embeddings of $(\Sigma,g)$ into $(\R^3,\delta)$ for $g \in C^4$. Here we need only up to one derivative of $K$, but we cannot get uniform bounds independent of $\e$. 

Also see Lu \cite{lu}, Theorem~1.3, for which only $C^{2,{\rm Dini}}$-bounds on $g$ are needed, subject to the assumption $K>2K_0$ where $(\M,\gbar)$ is the space form of constant curvature $K_0 \in \R$. 
\end{remark}

\begin{proof}
	 We consider the inverse of the mean curvature, namely
\begin{equation*}
W:=\frac{1}{H}.
\end{equation*}
Our goal is to prove that $$\min_\Sigma W \geq c>0$$ for some constant $c$ depending on $K$ and $dK$ only. This is achieved by establishing a Harnack estimate for $W$.

Dividing by $H^2$ on both sides of Eq.~\eqref{PDE, H}, we find that $W$ satisfies a first-order PDE:
\begin{equation}\label{PDE}
-d W \circ \jt = W\beta + \pi_\Sigma^\#\omega (1-4k^2W^2).
\end{equation}
See Labourie \cite{l}, 3.6 Proposition. 
The $1$-form $\beta$ is globally defined on the $1$-jet bundle, and it depends only on $\pi$, $k$, and the Riemann curvature of $\M$.

We shall consider $$\DD_0 :=\Big\{z \in \DD: W(z) < \delta k \Big\},$$ and our objective is to derive a lower bound for $W$ on $\DD_0$. Choose
\begin{equation*}
\delta := \frac{1}{8 \,(\sup_\Sigma k)^4}
\end{equation*}
so that $1-4k^2W^2 >1/2$. Note that $\delta$ is strictly positive by Gauss--Bonnet, as $(\Sigma,g)$ is a non-negatively curved topological $\stwo$.

Taking the exterior differential to Eq.\eqref{PDE},  one obtains
\begin{align}\label{2nd order PDE}
-d(dW \circ \jt)&= Wd\beta + dW \wedge \beta  + \pi^\#_\Sigma \Omega (1-4k^2W^2)  \nonumber\\
&\qquad + 8k^2W \pi^\#_\Sigma\omega  \wedge dk +  8kW^2 \pi^\#_\Sigma\omega  \wedge dW \nonumber\\
&=: \mathcal{S}(z, W, dW),
\end{align}
where $\Omega$ is the curvature form of $(\Sigma,g)$. Eq.~\eqref{2nd order PDE} is an identity of $2$-forms; see \cite{l}, p.409 Eq.~(2). %We employ here the structural equation \emph{\`{a} la} Cartan: $$\Omega = d\omega - \omega \wedge \omega.$$ 

By the definition for the  metric on subbundle $\V \subset \E$ (see Eq.~\eqref{herm metric} above), we have
\begin{eqnarray*}
\big|\pi^\#_\Sigma \Omega\big| \leq \sqrt{k^2 ( |\Omega|_\Sigma^2 + |J\Omega|_\Sigma^2 )}.
\end{eqnarray*}
Here the length $|\bullet|_\Sigma$ and the almost complex structure $J$ on $\Sigma$ are both computed \emph{with respect to the metric $\two$}. Thus we have
\begin{equation}\label{omega, Omega est}
\big|\pi^\#_\Sigma \Omega\big| \leq C_1 {k} W 
\end{equation}
for a uniform constant $C_1$ depending only on the $C^2$-norm of $g$ and $\gbar$.

We now substitute into Eq.~\eqref{2nd order PDE} the following relation 
\begin{equation*}
\pi^\#_\Sigma \omega = \frac{dW \circ \jt - W\beta}{1-4k^2W^2}.
\end{equation*}	
Again we need $1-4k^2W^2 >1/2$ on $\DD_0$ to make sense of this formula. Thus one obtains
\begin{align*}
\mathcal{S}(z, W, dW) &= Wd\beta + dW \wedge \beta + \pi^\#_\Sigma\Omega(1-4k^2W^2)  \\
&\qquad + \frac{8k^2W}{1-4k^2W^2} (dW \circ \jt-W\beta)\wedge dk \\
&\qquad+ \frac{8kW^2}{1-4k^2W^2}(dW \circ \jt-W\beta)\wedge dW.
\end{align*}
The above can be estimated pointwise:
\begin{align*}
\big|\mathcal{S}(z, W, dW) \big| &\leq |d\beta| W + |\beta| |dW| + C_1 {k} W + 4C_1k^{3}W^3\\
&\qquad + 16k^2 |dk| W  \big(|dW|+|\beta|W\big)\\
&\qquad + 16kW^2|dW|\big(|dW|+|\beta|W\big),
\end{align*}
thanks to Eq.~\eqref{omega, Omega est} and that $1-4k^2W^2 > 1/2$ on $\DD_0$.

As $0 \leq W < \delta k$ on $\DD_0$, we further bound at each $z \in \DD_0$ that
\begin{align*}
\big|\mathcal{S}(z, W, dW) \big| &\leq  |d\beta| W + |\beta| |dW| + C_1 k W + 4C_1 \delta^2 k^{5}W\\
&\qquad + 16\delta k^3|dk||dW| + 16 |\beta| \delta k^3 |dk|W \\
&\qquad + 16 \delta^2 k^3 |dW|^2 + 16\delta^3|\beta|k^4|dW|.
\end{align*}
Denote by 
\begin{eqnarray*}
&&\Lambda := \sup_{z\in\Sigma} k(z), \qquad
\Lambda_1 :=  \sup_{z\in\Sigma} k(z)|dk(z)|;\\
&&B := \sup_{z\in\Sigma} |\beta(z)|, \qquad
B_1 :=  \sup_{z\in\Sigma} |d\beta(z)|. 
\end{eqnarray*}
Thus on $\DD_0$ we have
\begin{equation}\label{structure, harnack}
\big|\mathcal{S}(z, W, dW) \big| \leq C_2|dW|^2 + C_3|dW| + C_4|W|,
\end{equation}
with the constants
\begin{eqnarray*}
&& C_2 = 16\delta^2 \Lambda^3,\\
&& C_3 = B+16\delta \Lambda^2 \Lambda_1 +16\delta^3 B \Lambda^4,\\
&& C_4 = B_1 + C_1{\Lambda} + 4C_1\delta^2 \Lambda^{5} + 16B\delta\Lambda^2\Lambda_1.
\end{eqnarray*}

To summarise, we have a second-order elliptic PDE of the divergence form $-d(dW \circ \jt) = \mathcal{S}$, namely Eq.~\eqref{PDE}. The left-hand side is simply the Laplace-Beltrami of $W$, due to the presence of the almost complex structure $\jt$. Thus, the bound~\eqref{structure, harnack} on the source term permits the application of the classical Harnack estimate; {\it cf.} Trudinger \cite{t}, Theorem~1.1 and Serrin \cite{s}. Indeed, for any cube $Q$ of edge length $3R$ inside the open set $\DD_0$, we have
\begin{equation}\label{harnack, estimate}
\min_Q W \geq C_5^{-1} \max_Q W,
\end{equation}
where $C_5$ depends on  $\mu$ and  $R\mu$;
\begin{equation*}
\mu = \sup_{\DD_0} \,(C_2+C_3+C_4).
\end{equation*}
Therefore, on $\DD_0$ we either  have $W\equiv 0$ constantly, or $W$ is non-vanishing everywhere.
\begin{itemize}
\item
In the former case, by the continuity of $W$ we know that the complement of $\DD_0$ must be empty; that is, $W \equiv 0$ on the whole domain $\DD$. But this is impossible in view of the following result due to Wang--Yau \cite{wy} and Shi--Tam \cite{st} (also see Lu \cite{lu}, Lemma~2.2):

\begin{lemma}\label{lem: wyst}
Let $(\Omega,\gbar)$ be a $3$-dimensional Riemannian manifold with scalar curvature $\overline{\rm scal}\geq -6 \kappa^2$ for some $\kappa>0$. Assume that $(\Sigma=\p\Omega,g)$ is a topological sphere with scalar curvature ${\rm scal}>-2\kappa^2$ and positive mean curvature $H$. Then $\Sigma$ can be isometrically embedded into $\mathbb{H}^3_{\kappa^2}$, the space form with constant negative sectional curvature $-\kappa^2$. 

Moreover, there holds
\begin{equation*}
\int_{\Sigma} (H_0-H) \cosh(\kappa r)\geq 0;
\end{equation*}
$r$ is the distance function on $\mathbb{H}^3_{\kappa^2}$  from the origin, and $H_0$ is the mean curvature of $\mathbb{H}^3_{\kappa^2}$.
\end{lemma}
It follows that the total curvature is uniformly bounded:
\begin{align*}
\int_\Sigma H \leq \int_\Sigma H\cosh(\kappa r) \leq \int_\Sigma H_0 \cosh(\kappa r) \leq C_6,
\end{align*}
where $C_6$ depends only on $\kappa$ and $\|g\|_{C^2}$. In addition, $\kappa$ can be chosen to depend only on $\|g\|_{C^3}$. This rules out the possibility that $H$ blows up everywhere on $\Sigma$; or, equivalently, that $W \equiv 0$ on the entire $\DD$.

\item
In the latter case, the Harnack estimate~\eqref{harnack, estimate} implies that $W \geq c_7 > 0$ on each $(3R)$-cube $Q \subset \DD_0$, where $c_7$ depends only on $\mu$ and $R$.

Fixing a small $R$ once and for all and applying a standard covering argument, we get 
\begin{equation*}
\min_{\DD_0} W \geq c_8 \max_{\DD_0} W
\end{equation*}
for some $c_8>0$ depending only on $\mu$.

On the other hand, by construction we have $W \geq \delta k$ on $\DD \sim \DD_0$.
\end{itemize}

To complete the proof, note that $\mu$ depends only on $B$, $B_1$, $\Lambda$, $\Lambda_1$, $C_1$, and $\delta$. Here $\Lambda_1$ and $B_1$ altogether depend on up to three derivatives of $g$ and $\gbar$. On the other hand, $C_1$, $B$, $\Lambda$, and $\delta$ depend on up to two derivatives of $g$ and $\gbar$. Furthermore, none of the above parameters depends on $\e$, {\it i.e.}, the lower bound for $k^2$. 

The assertion follows once we take $a_0=(c_8)^{-1}$ and $b_0=\delta^{-1}$. \end{proof}

\section{Proof of Theorem~\ref{thm: main}}\label{sec: proof}

In this section we 
deduce Theorem~\ref{thm: main} from the dichotomy Theorem~\ref{thm: weyl}.

\begin{proof}[Proof of Theorem~\ref{thm: main}]

First of all, as in Guan--Li \cite{gl}, let us approximate $g$ by a sequence of $C^\infty$-metrics $\{g^\e\}$, which both converges to $g$ in the $C^3$-topology and possesses $\e$-elliptic isometric immersions. One may take a conformal change of metrics $${g}^\e := e^{2\e\lambda}g$$ for a smooth scalarfield $\lambda$ on $\Sigma$. Indeed, as the Gauss curvature $\tilde{K}_\e$ for $(\Sigma,g^\e)$ satisfies
\begin{equation*}
-\e\Delta_g \lambda + \tilde{K} = \tilde{K}_\e e^{2\e\lambda},
\end{equation*}
where $\tilde{K}$ is the Gauss curvature of $g$, by imposing $-\Delta_g\lambda=1$ on $K^{-1}\{0\}$ we can ensure the strict inequality $\tilde{K}_\e > K_0$ everywhere on $\Sigma$. By Th\'{e}or\`{e}me A in Labourie \cite{l}, for each such $g_\e$ there exists an $\e$-elliptic isometric immersion $f_\e$ into $(\M,\bar{g})$. Furthermore, thanks to le lemme de Schwarz \`{a} la Gromov (\cite{l}, 1.2),  $f_\e$ is smooth for each $\e>0$.

In the sequel, let us denote by $H_\e$, $K_\e$, $\kappa_1^\e$, and  $\kappa_2^\e$ the mean curvature, the Gauss curvature, the smaller principal curvature, and the larger principal curvature for $f_\e$, respectively. Note that $K_\e$ is different from $\tilde{K}_\e$ in the last paragraph. All these quantities are non-negative, in view of the $\e$-ellipticity of $f_\e$. One also writes $k_\e:=\sqrt{K_\e}$. Note that $\{K_\e\}$ is uniformly bounded in $C^0$.

In the first case, assume that   $\{H_\e\}$ is uniformly bounded in $C^0$ by a constant depending only on the $C^3$-norms of $g$ and $\gbar$. We can pass to the limits to obtain a $C^{1,1}$-isometric immersion that is degenerate-elliptic, thanks to the Arzel\`{a}--Ascoli theorem.  %As discussed in Section~\ref{sec: Introduction}, by adapting \cite{l}, 1.6, one may infer that $f$ is either an isometric immersion or the zero map. But the latter is impossible: the Gauss curvature of $f$ has to be positive somewhere, thanks to the Gauss--Bonnet theorem (recall that $\Sigma$ is a topological $\stwo$). Thus we get a degenerate-elliptic isometric immersion. 

Now, let us suppose that $H$ blows up somewhere but not everywhere. For further developments, it is crucial to note that all the estimates in Theorem~\ref{thm: weyl} --- in particular, the constant $a_0$ and $b_0$ --- are independent of $\e$. So this theorem holds verbatim after replacing the data $\{f,g,H,K\equiv k^2,\kappa_1,\kappa_2\}$ by $\{f_\e, g^\e, H_\e, K_\e\equiv (k_\e)^2, \kappa_1^\e, \kappa_2^\e\}$, respectively.

Consider a point $z \in \Sigma$ such that $K_\e(z) \to 0$ as $\e \map 0$. The smaller eigenvalue $\kappa_1^\e(z)$ must tend to zero. By Theorem~\ref{thm: weyl} we have
\begin{equation*}
H_\e(z)k_\e(z) \leq b_0,
\end{equation*}
where $b_0$ is independent of $\e$. That is, $$\Big(\kappa_1^\e(z)+\kappa_2^\e(z)\Big)\sqrt{\kappa_1^\e(z)\kappa_2^\e(z)} \leq 2 b_0.$$ It follows that 
\begin{equation*}
\kappa_2^\e(z) \leq \Bigg(\frac{4(b_0)^2}{\kappa_1^\e(z)}\Bigg)^{\frac{1}{3}}
\end{equation*}
whenever $\kappa_1^\e(z) \neq 0$. Here we can choose $b_0=9(\sup_\Sigma k)^4\equiv 9(\sup_\Sigma K)^2$ for sufficiently small $\e$, thanks to Theorem~\ref{thm: weyl}.

On the other hand, if $z$ is not a point of degeneracy for the Gauss curvature, then  $\{H_\e(z)\}$ is uniformly bounded from the   above by $a_0$, which depends only on $\|g\|_{C^3}$ and $\|\bar{g}\|_{C^3}$ as in Theorem~\ref{thm: weyl}. Again, $a_0$ is independent of $z$. Then, utilising the na\"{i}ve bound
\begin{equation*}
H_\e \geq \sqrt{\frac{K_\e}{2}} = \frac{k_\e}{\sqrt{2}},
\end{equation*} 
it is straightforward to see that
\begin{equation*}
\kappa_2^\e(z) \leq \Bigg(\frac{{2H_\e(z) K_\e(z)}}{{\kappa^\e_1(z)}}\Bigg)^{\frac{1}{3}} \leq \frac{\sqrt[3]{4} a_0}{\sqrt[3]{\kappa_1^\e(z)}}.
\end{equation*}

The proof is now complete.  \end{proof}

\section{Concluding Remarks}\label{sec: rem}

A resolution for the classical Weyl problem  follows directly from the proof of Theorem~\ref{thm: weyl}. The essential ingredients of the proof are already present in Labourie \cite{l}.

\begin{corollary}\label{cor}
Let $(\Sigma,g)$ be a homeomorphic copy of $\stwo$ with Gauss curvature strictly greater than $K_0 \in \R$. Let $(\M,\gbar)$ be a simply-connected $3$-dimensional Riemannian manifold with sectional curvature strictly less than $K_0$. Assume that $g,\gbar \in C^3$. Then $(\Sigma,g)$ can be isometrically immersed in $(\M,\gbar)$ as a $C^3$-surface.
\end{corollary}

\begin{proof}
	The same arguments for Theorem~\ref{thm: weyl} yield that $W \geq c>0$ on $\Sigma$ unless $W \equiv 0$; the latter is again impossible due to Lemma~\ref{lem: wyst}. Here $c$ is allowed to depend on lower bound of $k^2$, which is strictly positive. On the other hand, all the relevant estimates only involve up to three derivatives of $g$ and $\gbar$. This gives a uniform bound on the mean curvature. Here we are working in the case of strict ellipticity, which is preserved under mollifications. Thus,  we can pass to the limits to establish the limiting isometric immersion via Arzel\`{a}--Ascoli's theorem and deduce the regularity from le lemme de Schwarz \`{a} la Gromov (\cite{l}, 1.1 and 1.2).  \end{proof}

Note that we only need $g,\gbar \in C^3$ here, which is a weaker assumption than that in Nirenberg \cite{n} ($C^4$), and Pogorelov \cite{p} and Lin \cite{lin} ($C^{3,\alpha}$). By different approaches Heinz \cite{h} also proved for $g \in C^3$. The above works also assume that $(\M,\gbar)$ is a space form. Recently, Lu \cite{lu} proved for $g \in C^{2,{\rm Dini}}$ by refining the estimates in \cite{h}.

On the other hand, we bring to the attention of the readers to 
the following problem of Guan--Li, which is concerned with sufficient conditions for the existence  \emph{smooth} isometric immersions/embeddings  (\cite{gl}, p.333 Question 2):
\begin{quotation}\emph{What are the sufficient conditions $($even necessary and sufficient conditions$)$ on the metric with nonnegative Gauss curvature which give rise to a smooth isometric embedding into $(\R^3,\delta)$? $[\delta$ is the Euclidean metric.$]$}\end{quotation}
It remains open in view of the counterexamples by  Burago--Shefel' and Iaia \cite{bs,i}.

\noindent
{\bf Acknowledgement}. The author is deeply indebted to Prof.~Pengfei Guan for many enlightening discussions on isometric immersions and the Weyl problem, to Prof.~Fang-Hua Lin for sharing his insights and ideas on Pogorelov's results and communications on the paper \cite{lin}, and to Prof.~Gui-Qiang Chen for his continuous support and lasting interests in our research. %SL also extends his gratitudes to Prof.~Siyuan Lu for insightful discussions,  for pointing out a mistake in an earlier version of the manuscript, and for bringing to our attention Lemma~\ref{lem: wyst} and uniform bounds on the total mean curvature. 

\end{document}